\documentclass{amsart}

\newtheorem{theorem}{Theorem}[section]
\newtheorem{proposition}{Proposition}[section]
\newtheorem{lemma}[theorem]{Lemma}

\theoremstyle{definition}
\newtheorem{definition}[theorem]{Definition}

\theoremstyle{remark}

\numberwithin{equation}{section}

\usepackage{amsthm}
\usepackage{amsfonts, ulem}
\usepackage{amsmath}
\usepackage{amssymb}
\usepackage{tikz}
\usepackage{multicol}
\usepackage[utf8]{inputenc}
\usetikzlibrary{matrix}
\usepackage{enumitem}
\usepackage{color}
\usepackage[all]{xy}

\def\ve{\varepsilon}

\def\R{{\mathbb R}}

\def\C{{\mathbb C}}

\def\N{{\mathbb N}}

\def\Z{{\mathbb Z}}

\def\E{{ \mathcal E}}

\def\su{_{\scriptscriptstyle U}}

\def\pr{^{\scriptscriptstyle \R}}
\def\po{^{\scriptscriptstyle O}}
\def\pu{^{\scriptscriptstyle U}}

\def\sr{_{\scriptscriptstyle \R}}
\def\sc{_{\scriptscriptstyle \C}}
\def\crr{^{\scriptscriptstyle {\it CR}}}
\def\crt{^{\scriptscriptstyle {\it CRT}}}

\def\CRT{\mathcal{CRT}}
\def\CR{\mathcal{CR}}

\newcommand{\smv}[2]
	{ \left( \begin{smallmatrix} {#1}  \\ {#2} \end{smallmatrix} \right)  }

\def\id{\text {id} \,}

\def\coker{\text {coker} \,}

\begin{document}

\title[Exotic Cuntz Algebras]{The Stable Exotic 
Cuntz Algebras are Higher-Rank Graph Algebras} 
\author{Jeffrey L. Boersema}
\address{Seattle University}
\email{boersema@seattleu.edu}
\author{Sarah L. Browne}
\address{University of Kansas}
\email{slbrowne@ku.edu}
\author{Elizabeth Gillaspy}
\address{University of Montana}
\email{elizabeth.gillaspy@mso.umt.edu}


\subjclass[2020]{46L80}

\date{\today}

\dedicatory{}

\commby{}

\begin{abstract}
For each odd integer $n \geq 3$, we construct a rank-3 graph  $\Lambda_n$ with involution $\gamma_n$ whose real $C^*$-algebra $C^*\sr(\Lambda_n, \gamma_n)$ is stably isomorphic to the exotic Cuntz algebra $\mathcal E_n$.  This construction is optimal, as we prove that a rank-2 graph with involution $(\Lambda,\gamma)$ can never satisfy $C^*\sr(\Lambda, \gamma)\sim_{ME} \mathcal E_n$, and 
the first author 
reached the same conclusion for rank-1 graphs (directed graphs)  in \cite[Corollary 4.3]{boersema-MJM}.  Our construction relies on a rank-1 graph with involution $(\Lambda, \gamma)$ whose real $C^*$-algebra $C^*\sr(\Lambda, \gamma)$ is stably isomorphic to the suspension $ S \R$.  In the Appendix, we show that the $i$-fold suspension $S^i \R$ is stably isomorphic to a graph algebra iff $-2 \leq i \leq 1$.
\end{abstract}

\maketitle


\bibliographystyle{amsplain}

\section{Introduction}
The complex $C^*$-algebras now known as the {Cuntz algebras} $\mathcal O_n$ were introduced in \cite{cuntz} as the first concrete examples of separable, simple, purely infinite $C^*$-algebras.  The Cuntz algebras quickly became central players in $C^*$-algebra theory, and have also been profitably  employed in a broad range of applications, from wavelets \cite{bratteli-jorgensen-memoirs} to duality for compact groups \cite{doplicher-roberts-cuntz-duality} to mathematical physics \cite{abe-kawamura}.

For every odd integer $n \geq 3$, 
there are two real $C^*$-algebras whose complexification is $\mathcal O_n$: 
the real Cuntz algebra $\mathcal O_n\pr$ and the exotic Cuntz algebra $\mathcal E_n$. 
While the existence of $\mathcal E_n$ follows from the classification of simple purely infinite real $C\sp*$-algebras  \cite{boersema-JFA, brs}, the non-constructive nature of the existence portion of this classification theorem \cite[Theorem~1]{boersema-JFA} means that we know very little about  $\mathcal{E}_n$ beyond its $K$-theory.  In particular, until now there has been no 
description of $\mathcal E_n$ in terms of familiar $C^*$-algebraic objects.  

In this paper, we give an explicit realization of the stabilized exotic Cuntz algebras $\mathcal K\sr \otimes\sr \mathcal E_n$ as higher-rank graph algebras associated to rank-3 graphs with involution.  Given the extensive literature on the properties of higher-rank graph $C^*$-algebras, we anticipate that this concrete description will facilitate an improved understanding of these elusive algebras.

Higher-rank graphs, or $k$-graphs, are a $k$-dimensional generalization of directed graphs which were introduced by Kumjian and Pask in \cite{kp}. Directed graphs are a key tool in $C^*$-algebra theory, because many  properties of their associated (complex) $C^*$-algebras, such as their $K$-theory \cite{raeburn-szyman} and  ideal structure \cite{bhrs,hong-szyman}, are visible from the graph. While the structure of $k$-graph $C^*$-algebras is more intricate than that of graph $C^*$-algebras, $k$-graph $C^*$-algebras also encompass a broader range of examples.  Indeed \cite{ruiz-sims-sorensen}, every complex UCT Kirchberg algebra is a direct limit of 2-graph $C^*$-algebras.
The real $C^*$-algebra $C^*\sr(\Lambda, \gamma)$ of a higher-rank graph with involution $(\Lambda, \gamma)$ was recently introduced by 
the first and third authors
 in \cite{boersema-gillaspy}.  In that paper, the authors also generalized the work of \cite{evans} and \cite{boersema-MJM} to describe a spectral sequence  which converges to the $\mathcal{CR}$ $K$-theory of these real $C^*$-algebras.

The main result (Theorem \ref{rank3existance}) of the present paper, that the exotic Cuntz algebra is stably isomorphic to the $C^*$-algebra of a 3-graph with involution,   is the best possible in terms of the rank of $\Lambda$. 
In \cite{boersema-MJM}, 
the first author 
made an extensive analysis of the $K$-theory of the real $C^*$-algebra $ C \sp * \sr(\Lambda, \gamma)$ of a rank-1 graph  (directed graph) with involution.\footnote{This class of $C^*$-algebras includes the real $C^*$-algebras $C^*\sr (\Lambda)$ of a directed graph, introduced in \cite{boersema-RMJ}, as $C^*\sr(\Lambda) \cong C^*\sr(\Lambda, \gamma_{\text{triv}})$.}   In particular, \cite[Corollary 4.3]{boersema-MJM}  establishes that the exotic Cuntz algebra cannot be (stably) isomorphic 
 to the  real $C \sp *$-algebra  $ C \sp * \sr(\Lambda)$ of a directed graph $\Lambda$, or to the real $C^*$-algebra $ C \sp * \sr(\Lambda, \gamma)$ of a graph  with involution, since $KO_7(\mathcal E_n) = \Z_2$ but $KO_7(C^*\sr(\Lambda,\gamma))$ is always torsion-free. Theorem \ref{thm:no-2-graph} below uses the $K$-theory spectral sequence for real higher-rank graph $C^*$-algebras (\cite[Section 3]{boersema-gillaspy}) to show that $\mathcal E_n \not \sim_{ME} C^*\sr(\Lambda, \gamma))$ for any rank-2 graph with involution $(\Lambda, \gamma)$. However, we construct in Theorem \ref{rank3existance} a family of rank-3 graphs with involution $\{(\Lambda_n, \gamma_n)\}_{n\geq 1}$ such that $C^*\sr(\Lambda_n, \gamma_n) \cong \mathcal E_{2n+1} \otimes\sr \mathcal K\sr$.
 
 To be precise, $\Lambda_n$ 
is a product graph, $\Lambda_n = E_n \times \Lambda \times \Lambda$.
The graph $E_n$ was introduced in 
\cite[Example 6.2]{boersema-MJM} and admits an involution $\varepsilon_n$ such that $C^*\sr(E_n, \ve_n) \sim_{ME} S^6 \mathcal E_{2n+1}$.  We describe the graph $\Lambda$ and involution $\gamma$ in Proposition \ref{suspension}, and show that $C^*\sr(\Lambda, \gamma) $ is a real Kirchberg algebra which is $KK$-equivalent to the suspension algebra $S \R \cong C_0((0,1), \R)$.  The K\"unneth Theorem for real $C^*$-algebras \cite{boersema2002} and the classification of real Kirchberg algebras \cite{brs} imply that $C^*\sr(\Lambda_n, \gamma_n) \cong \mathcal K\sr \otimes\sr \mathcal E_{2n+1}$.

Prompted by the graph with involution of Proposition \ref{suspension}, we consider in Section \ref{sec:appendix} the question of which suspensions $S^i \R$ are $KK$-equivalent to the real $C^*$-algebra of a graph with involution. For $-2 \leq i \leq 1$ we exhibit an example of a graph with involution $(\Lambda, \gamma)$ such that $C^*\sr(\Lambda, \gamma)$ is $KK$-equivalent to $S^i\R$, and we show in Proposition \ref{prop:last} that $S^i\R\not\sim_{KK} C^*\sr(\Lambda, \gamma)$ if $2 \leq i \leq 5$.  (However, we can realize these suspensions as 2-graph or 3-graph algebras, by taking products of the graphs which do realize suspensions of $\R$.)

Many key questions remain open for further investigation about the class of real $C \sp *$-algebras that can be obtained using higher-rank graphs. For example, it is still unknown whether or not $\mathcal{E}_n$ itself can be realized as a rank-$k$ graph-with-involution algebra. Similarly, it remains unknown which real Kirchberg algebras can be realized by higher-rank graphs with involution (as opposed to inductive limits of such objects); we would particularly like to find a $K$-theoretic characterization of such algebras.

{\bf Acknowledgments:} E.G.
was partially supported by NSF grant 1800749.

\section{Preliminaries}

\subsection{Higher-rank graphs}
\label{sec:k-graph}
\begin{definition} \cite[Definition 1.1]{kp}
A {\it higher-rank graph} of rank $k$, or a {\it $k$-graph}, is a countable small category $\Lambda$ equipped with a degree functor $d \colon \Lambda \rightarrow \mathbb{N}^k$ such that, if a morphism $\lambda \in \Lambda$ satisfies $d(\lambda)=m+n$, then there exist unique morphisms $\mu, \nu \in \Lambda$ such that $\lambda=\mu\nu$, $d(\mu)=m$ and $d(\nu)=n$. 
\end{definition}
Write $e_i$ for the standard $i$th basis vector of $\N^k$.
The morphisms of degree $e_i$ can be advantageously viewed as the ``edges of color $i$'' in $\Lambda$.  In this perspective, if $e$ is an edge of color $i$ and $f$ is an edge of color $j$, their composition $ef \in \Lambda$ satisfies 
\[ d(ef) = e_i + e_j = e_j + e_i,\]
so we must be able to rewrite $ef = f'e'$ for some morphisms $e', f'\in \Lambda$ with $d(f') = e_j$ and $d(e') = e_i$.  

Indeed, by \cite[Theorems 4.4 and 4.5]{hazle-raeburn-sims-webster}, a $k$-graph can be equivalently thought of as arising from a directed graph $G$, with $k$ colors of edges and with a factorization rule on multicolored paths. 
That is, given any two colors (``red'' and ``blue'') and any two vertices $v, w$ in $G$, the factorization rule identifies each red-blue path $ef$ from $v$  to $w$ with an equivalent blue-red path $f'e'$  from $v$ to $w$.  

We would like  the quotient of the space $G^*$ of directed paths in $G$ by the equivalence relation $\sim$ generated by the factorization rule to be a $k$-graph.  For this to occur, 
the factorization rule must also satisfy certain consistency conditions which ensure that, for each path in $G^*$, its equivalence class  under $\sim$ corresponds to a $k$-dimensional hyper-rectangle; see \cite[Theorem 2.3]{EFGGGP} for more details.  As our work in this paper does not depend on these consistency conditions, we will not reproduce them here.  That said, we remark that in a rank-1 graph, the factorization rule  is nonexistent, and so a 1-graph is precisely the space of paths of a directed graph.

Let $\Lambda$ be a $k$-graph.  Given $n\in \N^k$ and objects $v,w \in \Lambda$, we write 
\begin{equation}
\Lambda^n = \{ \lambda \in \Lambda: d(\lambda) = n\}.
\label{eq:kgraph-notation-degree}
\end{equation}  
By the factorization rule, for every $\lambda \in \Lambda$, there are unique $v, w \in \Lambda^0$ with $v \lambda = \lambda w = v\lambda w = \lambda$.  That is, we can identify $\Lambda^0$ with the objects of $\Lambda$. If $\lambda = v \lambda w$, we write $v = r(\lambda)$ and $w = s(\lambda)$.  Thus, expanding on Equation \eqref{eq:kgraph-notation-degree}, we have 
\[\begin{aligned}
 v\Lambda^n &= \{ \lambda \in \Lambda: r(\lambda) = v \text{ and } d(\lambda) = n\} \text{ and }
  \Lambda^n w &= \{ \lambda \in \Lambda: s(\lambda) = w \text{ and } d(\lambda) = n\},
\end{aligned} \]
as well as the obvious variations such as $v\Lambda^n w$.

A $k$-graph $\Lambda$ has $k$ {\it adjacency matrices} $M_i \in M_{\Lambda^0}(\N)$, which are given by 
\begin{equation}
\label{eq:adjacency-matrices}
M_i(v,w) = \# v\Lambda^{e_i} w.
\end{equation}

 In the graphical picture, $\lambda \in \Lambda^{(n_1, \ldots, n_k)}$ means that $\lambda$ represents the $\sim$-equivalence class of a path with $n_i$ edges of color $i$, for each $1\leq i \leq k$.  That is, $\Lambda^0$ consists of the length-0 paths, ie, the vertices. Then $M_i(v,w)$ is the number of edges of color $i$ from vertex $w$ to vertex $v$.

If $\Lambda_1$ is a $k_1$-graph and $\Lambda_2$ is a $k_2$-graph, then \cite[Proposition 1.8]{kp}   their (Cartesian) product $\Lambda_1 \times \Lambda_2$ is a $(k_1 + k_2)$-graph; the degree functor is given by $d(\lambda_1, \lambda_2) = (d(\lambda_1), d(\lambda_2)$.  We have $(\Lambda_1 \times \Lambda_2)^0 = \Lambda_1^0 \times \Lambda_2^0$ and $s(\lambda_1 \times \lambda_2) = (s(\lambda_1), s(\lambda_2))$.

In this paper we will focus on $k$-graphs which are {\it row-finite} and {\it source-free} (or {\it have no sources}).
We say a $k$-graph  $\Lambda$ is row-finite if $|v\Lambda^n|<\infty$ for all $n \in \mathbb{N}^k$ and $v \in \Lambda^0$. 
The $k$-graph has no sources if $v\Lambda^n \not= \emptyset$ for all $v,n$.
It is straightforward to check that if $\Lambda_1, \Lambda_2$ are row-finite and source-free, then so is $\Lambda_1 \times \Lambda_2$.

For a row-finite source-free $k$-graph $\Lambda$, its (complex) $C^*$-algebra is the universal $C^*$-algebra generated by a Cuntz--Krieger $\Lambda$-family.  

\begin{definition}\cite[Definition 1.5]{kp}
Given a row-finite source-free $k$-graph $\Lambda$, a {\it Cuntz--Krieger $\Lambda$-family} is a collection $\{t_\lambda\}_{\lambda \in \Lambda}$ of partial isometries  in a $C\sp*$-algebra $A$ which satisfy the following conditions:
\begin{enumerate}
\item[(CK1)] For each $v \in \Lambda^0$, $t_v$ is a projection, and $t_v t_w = \delta_{v,w} t_v$.
\item[(CK2)] For each $\lambda \in \Lambda, \ t_\lambda^* t_\lambda = t_{s(\lambda)}$.
\item[(CK3)] For each $\lambda, \mu \in \Lambda, \ t_\lambda t_\mu = t_{\lambda \mu}$.
\item[(CK4)] For each $v \in \Lambda^0$ and each $n \in \N^k$, 
$\displaystyle t_v = \sum_{\lambda \in v\Lambda^n} t_\lambda t_\lambda^*.$
\end{enumerate}
We define $C\sp*(\Lambda)$ to be the universal (complex) $C\sp*$-algebra generated by a Cuntz--Krieger family, in the sense that for any Cuntz--Krieger $\Lambda$-family $\{t_\lambda\}_{\lambda \in \Lambda}$, there is a surjective $*$-homomorphism $C\sp*(\Lambda) \to C\sp*(\{ t_\lambda\}_\lambda)$.  
\end{definition}

We write $\{s_\lambda\}_{\lambda \in \Lambda}$ for the generators of $C\sp*(\Lambda)$.
The Cuntz--Krieger relations imply that $C\sp*(\Lambda) = \overline{\text{span}} \{ s_\lambda s_\mu^*: s(\lambda) = s(\mu)\}.$
By 
\cite[Corollary 3.5(iv)]{kp},  the map $s_{(\lambda_1, \lambda_2)} $ to $s_{\lambda_1} \otimes s_{\lambda_2}$ gives an isomorphism $C^*(\Lambda_1 \times \Lambda_2) \cong C^*(\Lambda_1) \otimes C^*(\Lambda_2)$,.

We will use one more ingredient -- an involution -- to construct the real $C^*$-algebras associated to higher-rank graphs.

\begin{definition}
An {\it involution} $\gamma$ on a $k$-graph $\Lambda$ is a degree-preserving functor $\gamma \colon \Lambda \rightarrow \Lambda$ which satisfies $\gamma \circ \gamma = \id_{\Lambda}$.
\end{definition}

By \cite[Definition 2.4]{boersema-gillaspy}, the real $C^*$-algebra associated to a $k$-graph $\Lambda$ and an involution $\gamma: \Lambda \to \Lambda$ is 
\begin{equation}
\label{eq:real-C*-alg}
C^*\sr(\Lambda, \gamma) =  \{a \in C \sp *(\Lambda) \mid \widetilde \gamma(a) = a^* \} ,
\end{equation}
where $\widetilde \gamma$ is the antimultiplicative $C^*$-involution uniquely determined by $\widetilde \gamma(s_\lambda) = s^*_{ \gamma(\lambda)}.$
For any involution $\gamma$ on $\Lambda$, $C^*\sr (\Lambda, \gamma)$ is a real form of $C^*(\Lambda)$: that is, $C^*(\Lambda) = \C \otimes\sr C^*\sr(\Lambda, \gamma)$.

\subsection{$\mathcal{CRT}$ $K$-theory}
\label{sec:CRT}
In our work, we will use the full united $K$-theory $K\crt(A)$ (introduced in \cite{boersema2002}) as well as the abbreviated variation $K\crr(A)$ which contains just the real and complex parts. Theorem~10.2 of \cite{brs} shows that the category of real purely infinite simple $C \sp *$-algebras, whose complexifications are simple and in the UCT class, is classified up to isomorphism by either of these invariants. We tend to use $K\crr(A)$ since it is simpler and usually sufficient, but we will also need to use $K\crt(A)$ on occasion since that is the context in which we have the K\" unneth formula.
Specifically, recall that for a real $C \sp *$-algebra $A$,
\begin{align*}
K\crr(A) &= \{ KO_* (A), KU_*(A)\} \quad  \text{and} \quad K\crt(A)  = \{ KO_*(A), KU_*(A), KT_*(A) \} \; , 
\end{align*}
where $KO_*(A)$ is the standard 8-periodic real $K$-theory for a real $C \sp *$-algebra and $KU_*(A) = K_*(\C \otimes\sc A)$ is the 2-periodic $K$-theory of the complexification of $A$. Meanwhile $KT_*(A)$ is the 4-periodic self-conjugate $K$-theory. These invariants also include the additional $\mathcal{CR}$ and $\mathcal{CRT}$-module structure. In particular for $K\crr(A)$ there are natural transformations 
\begin{align*}
r_i &\colon KU_i(A) \rightarrow KO_i(A) && \text{induced by the standard inclusion }  \mathbb{C} \rightarrow M_2(\mathbb{R})  \\
c_i & \colon KO_i(A) \rightarrow KU_i(A) && \text{induced by the standard inclusion }  \mathbb{R} \rightarrow \mathbb{C} \\
\psi_i &\colon KU_i(A) \rightarrow KU_i(A) &&  \text{induced by conjugation }  \mathbb{C} \rightarrow \mathbb{C}  \\
\eta_i &\colon KO_i(A) \rightarrow KO_{i+1}(A) &&  \text{induced by multiplication by $\eta \in KO_1(\R) = \Z_2$} \\
\end{align*}

The additional structure tends to aid in the computations of $KO_*(A)$ because the natural transformations satisfy the relations
\begin{align*}
&rc=2 && cr=1+\psi && 2\eta =0 \\
&r\psi=r && \psi^2=\id && \eta^3=0 \\
&\psi c=c && \psi\beta\su =-\beta\su\psi && \xi = r\beta\su^2c \\
\end{align*}
and they fit into a long exact sequence
\begin{equation}
\cdots \xrightarrow{r\beta\su^{-1}} KO_i(A) \xrightarrow{\eta} KO_{i+1}(A) \xrightarrow{c} 
	KU_{i+1}(A) \xrightarrow{r\beta\su^{-1}} KO_{i-1}(A) \xrightarrow{\eta} \cdots 
	\label{eq:LES}
	\end{equation}

These two invariants $K\crr(A)$ and $K\crt(A)$ contain the same essential information by results of \cite{hewitt} (also summarized in \cite[Proposition~2.5]{brs}). 

\subsection{$K$-theory for higher-rank graphs}
\label{sec:k-graph-K-thy}
For the real $C \sp *$-algebra $C \sp * \sr(\Lambda, \gamma)$ of a higher-rank graph with involution, \cite[Theorem 3.13]{boersema-gillaspy} establishes the existence of a spectral sequence $\{E^r, d^r \}$ of $\CR$-modules that converges to $K\crr (C \sp * \sr(\Lambda, \gamma))$. The complex part of this spectral sequence 
$(E^r_{p,q})\pu$ coincides with the Evans spectral sequence \cite{evans} and converges to 
$KU_*(C \sp * \sr (\Lambda, \gamma)) = K_*( C \sp *(\Lambda))$. 
The real part of this spectral sequence 
$(E^r_{p,q})\po$ converges to 
$KO_*(C \sp * \sr (\Lambda, \gamma))$.

The $E^2$ page of the spectral sequence arises from the homology of a certain chain complex $\mathcal{C}$ based on the combinatorial information of $\Lambda$ and $\gamma$. We will use the spectral sequence only in the rank-1 and rank-2 cases, where these chain complexes have the following straightforward descriptions (cf.~\cite[Theorems 3.17 and  3.18]{boersema-gillaspy}).
Fix a partition of the vertices $\Lambda^0 = \Lambda^0_f \sqcup \Lambda^0_g \sqcup \Lambda^0_h$ where $\Lambda^0_f$ is the set of fixed vertices and 
$\gamma(\Lambda^0_g) = \Lambda^0_h$. 
Then we set $\mathcal{A}=K^{\crr}(\mathbb{R})^{\Lambda^0_f} \oplus K^{\crr}(\mathbb{C})^{\Lambda^0_g} $. 

For a rank-1 graph with involution the chain complex $\mathcal{C}$ is given by
\[ 0 \rightarrow \mathcal{A} \xrightarrow{\partial_1} \mathcal{A} \rightarrow 0, \]
where $\partial_1 = \rho^1$. 
For a rank-2 graph with involution the chain complex $\mathcal{C}$ is given by
\[ 0 \rightarrow \mathcal{A} \xrightarrow{\partial_2} \mathcal{A}^2 \xrightarrow{\partial_1} \mathcal{A} \rightarrow 0, \]
where
$\partial_1 = \begin{pmatrix} \rho^1 & \rho^2 \end{pmatrix}$ and $
\partial_2 = \smv{- \rho^2}{\rho^1}$.
Here the maps $\rho^i$ are entirely determined by the adjacency structure of $\Lambda$.
For $1\leq i \leq k$, the complex part $(\rho^i)_0\pu \colon \Z^{\Lambda^0} \rightarrow \Z^{\Lambda^0}$ is represented by the matrix $B_i = I - M^t_i$, 
where $M_i$ is the adjacency matrix of the graph $\Lambda$ for the  edges of degree $e_i$, and $(\rho^i)_1\pu = 0$ for all $i$. 
The real parts of this map $(\rho^i)_j\po$, for $0 \leq j \leq 7$, can similarly be determined for each $i$ by some variations of $B_i$ as shown in Table~3 of \cite{boersema-gillaspy} and Theorem~4.4 of \cite{boersema-MJM}.
%

\vspace{0mm}
{For reference, the groups of $K\crr(\R)$ and $K\crr(\C)$ are shown below. The natural transformations $\eta$, $c$, $r$, and $\psi$ that are part of the structure of united $K$-theory are uniquely determined from these groups and the long exact sequence \eqref{eq:LES}; they are also shown in Tables~1 and 2 of \cite{boersema-gillaspy}.
In particular we note that for $KO_*(\R)$, the map $\eta_i$ is non-trivial exactly for $i = 0,1$. 

\[\begin{array}{|c|c|c|c|c|c|c|c|c|}  
\hline  \hline  
  & \makebox[.8cm][c]{0} & \makebox[.8cm][c]{1} & 
\makebox[.8cm][c]{2} & \makebox[.8cm][c]{3} 
& \makebox[.8cm][c]{4} & \makebox[.8cm][c]{5} 
& \makebox[.8cm][c]{6} & \makebox[.8cm][c]{7} 
 \\
\hline  \hline 
KO_* (\R) & 
{\Z} & {\Z_2} &
{\Z_2} & {0} & 
{\Z} &{0} &{0} & {0}      \\
\hline  
KU_* (\R)  &
\Z & 0 & 
\Z & 0 & 
\Z & 0 & 
\Z & 0         \\
\hline \hline
\end{array} \]
\[\begin{array}{|c|c|c|c|c|c|c|c|c|}  
\hline  \hline  
  & \makebox[.8cm][c]{0} & \makebox[.8cm][c]{1} & 
\makebox[.8cm][c]{2} & \makebox[.8cm][c]{3} 
& \makebox[.8cm][c]{4} & \makebox[.8cm][c]{5} 
& \makebox[.8cm][c]{6} & \makebox[.8cm][c]{7} 
 \\
\hline  \hline 
KO_* ( \C) &
\Z & 0 & 
\Z & 0 & 
\Z & 0 & 
\Z & 0         \\
\hline  
KU_* ( \C )  &
\Z^2 & 0 & 
\Z^2 & 0 & 
\Z^2 & 0 & 
\Z^2 & 0         \\
\hline \hline
\end{array} \]
}

\section{Non-Existence of a Rank-2 Graph with Involution} \label{lowerbound}

\begin{theorem}
\label{thm:no-2-graph}
Let $n \geq 3$ be odd. There does not exist a row-finite, source-free rank-2 graph with involution $(\Lambda, \gamma)$ such that 
$K\crr( C \sp * \sr( \Lambda, \gamma)) \cong K\crr( \mathcal{E}_n )$.
\end{theorem}

\begin{proof}
Suppose that $(\Lambda, \gamma)$ is a row-finite, source-free rank-2 graph with involution  
such that
$K\crr( C \sp * \sr( \Lambda, \gamma)) \cong K\crr( \mathcal{E}_n )$. 
The groups of $K\crr(\mathcal{E}_n )$ are as follows:
\[\begin{array}{|c|c|c|c|c|c|c|c|c|}  
\hline  \hline  
  & \makebox[.8cm][c]{0} & \makebox[.8cm][c]{1} & 
\makebox[.8cm][c]{2} & \makebox[.8cm][c]{3} 
& \makebox[.8cm][c]{4} & \makebox[.8cm][c]{5} 
& \makebox[.8cm][c]{6} & \makebox[.8cm][c]{7} 
 \\
\hline  \hline 
KO_* ( \mathcal{E}_n ) & {\Z_{2(n-1)}} & {\Z_2} &
{\Z_2} & {0} & 
{\Z_{(n-1)/2}} &{0} &{\Z_2} & {\Z_2}      \\
\hline  
KU_* ( \mathcal{E}_n)  &
\Z_{n-1} & 0 & 
\Z_{n-1} & 0 & 
\Z_{n-1} & 0 & 
\Z_{n-1}  & 0         \\
\hline \hline
\end{array} \]
Note that since $KU_7( \mathcal{E}_n ) = 0$ and since $\text{im}\, \eta_6 = \ker c_7$ from the long exact sequence \eqref{eq:LES} relating $KO_*(A)$ and $KU_*(A)$, we see immediately that $\eta_6 \colon \Z_2 \rightarrow \Z_2$ must be an isomorphism.

Now, we consider only the real part of the spectral sequence from  \cite[Theorem~3.18]{boersema-gillaspy}.
This is a spectral sequence converging to $KO_*( C \sp * \sr( \Lambda, \gamma))$, where the $E^2$-page consists of the homology of the chain complex
\begin{equation}
0 \rightarrow \mathcal A \xrightarrow{~\partial_2~} \mathcal A^2 \xrightarrow{~\partial_1~} \mathcal A \rightarrow 0  \; 
\label{eq:chain-complex-deg2}
\end{equation}
where  
$\mathcal A \cong KO_*(\R)^{\Lambda^0_f} \oplus KO_*(\C)^{\Lambda^0_g}$ as described in Section~\ref{sec:k-graph-K-thy}.  
In particular, the spectral sequence has three non-zero columns (for $0 \leq p \leq 2$) and is periodic in $q$ (with period 8). Furthermore, a quick examination of the structure 
of the structure of the chain complex \eqref{eq:chain-complex-deg2} 
 reveals that $\mathcal A_i = 0$ for $i = 3,5,7$ and also that $\mathcal A_i$ is free for $i = 0, 4,6$.  Consequently,  $E^2_{p,q}  = 0$ for all $p$ and for $q = 3,5,7$. 
Moreover, since  $E^2_{2,q} = \ker \partial_2 \leq \mathcal A_q$, 
$E^2_{2,q}$ is free for $q = 0,4,6$. 
Indeed, $E^\infty_{2, q} = \ker d^2_{2,q} \leq E^2_{2,q}$, so $E^\infty_{2,q}$ must also be free for $q = 0,4,6$.
On the other hand,
$KO_i( C \sp * \sr (\Lambda, \gamma))$ is finite in all degrees, so since $E^\infty_{2,q}$ is a quotient of $KO_q(C^*\sr(\Lambda, \gamma))$,  $E^\infty_{2,q} = 0$ for $q = 0,4,6$. 

The above remarks imply that the  $E^\infty_{p,q}$ page of the spectral sequence is as follows:
\[ \begin{array}{ |c|ccc| }
\hline
~  & \hspace{.1cm} \vdots \hspace{.1cm} &\hspace{.1cm} \vdots \hspace{.1cm} & \hspace{.1cm} \vdots \hspace{.1cm} \\ 
7   & 0 & 0 & 0  \\
6     &  *    &  *   &  0    \\
5     &  0    &  0   &  0    \\
4     &  * & * & 0 \\
3     &  0  & 0 &  0   \\
 2    &  * & * & * \\   
 1     &  * & * & *   \\
0   &  * & * & 0  \\ \hline
q/p  & 0 & 1 & 2 \\ \hline
\end{array}  \] 
This $E^\infty$ page identifies a filtration of $KO_*(C^*\sr(\Lambda, \gamma))$, in which the subquotients of $KO_j(C^*\sr(\Lambda,\gamma))$ appear along the diagonal $p+q = j$ of the $E^\infty$ page. 

By hypothesis we have
$KO_0( C \sp * \sr (\Lambda, \gamma)) = \Z_{2(n-1)}$. 
However, the only non-zero group along the diagonal $p+q =0$ is $E^\infty_{0,0}$. Thus $E^\infty_{0,0} = \Z_{2(n-1)}$.
Similarly, since  $KO_7( C \sp * \sr (\Lambda, \gamma)) = \Z_2$ and $KO_6( C \sp * \sr (\Lambda, \gamma)) = \Z_2$,
we must have $E^\infty_{1,6} = \Z_2 = E^\infty_{0,6}$: 
\[ \begin{array}{ |c|ccc| }
\hline
~  & \hspace{.1cm} \vdots \hspace{.1cm} &\hspace{.1cm} \vdots \hspace{.1cm} & \hspace{.1cm} \vdots \hspace{.1cm} \\ 
7   & 0 & 0 & 0  \\
6     & \Z_2    &  \Z_2   &  0    \\
5     &  0    &  0   &  0    \\
4     &  * & * & 0 \\
3     &  0  & 0 &  0   \\
 2    &  * & * & * \\   
 1     &  * & * & *   \\
0   &  \Z_{2(n-1)} & * & 0  \\ \hline
q/p  & 0 & 1 & 2 \\ \hline
\end{array}  \] 

Now, we consider the natural transformation $\eta \colon \mathcal A \rightarrow \mathcal A$ of degree 1. Because $\mathcal A$ is a direct sum of copies of $K\crr(\R)$ and $K\crr(\C)$, {which data already includes the map $\eta$}, the natural transformation $\eta: KO_*(C^*\sr(\Lambda, \gamma)) \to KO_{*+1}(C^*\sr(\Lambda, \gamma))$ exists at the level of the chain complex \eqref{eq:chain-complex-deg2}, passes to a map on the $E^2$-page, then to a map on the $E^\infty$-page, and finally converges to the map
$\eta \colon KO_i( C \sp * \sr (\Lambda, \gamma)) \rightarrow KO_{i+1} (C \sp * \sr (\Lambda, \gamma))  $
described in Section \ref{sec:CRT}.
This means that the map $\eta$ on $KO_*(C \sp * \sr(\Lambda, \gamma))$
respects the filtration of $KO_i( C \sp * \sr (\Lambda, \gamma))$ associated with the spectral sequence; and the resulting maps on the subquotients are the same as that on the $E^\infty$ page.
In particular, the diagonals of the $E^\infty$-page that yield $KO_6(C \sp * \sr (\Lambda, \gamma))$ and $KO_7(C \sp * \sr (\Lambda, \gamma))$ give the  commutative diagram 
\[ \xymatrix{
	0  \ar[r] 
	& E^\infty_{0.6} \ar[r] \ar[d]^{\eta} 	
	& KO_6( C \sp * \sr (\Lambda, \gamma)) \ar[r] \ar[d]^\eta
	& E^\infty_{1,5} \ar[r] \ar[d]^\eta
	& 0 \\
	0  \ar[r] 
	& E^\infty_{0.7} \ar[r] 	
	& KO_7( C \sp * \sr (\Lambda, \gamma)) \ar[r] 
	& E^\infty_{1,6} \ar[r] 
	& 0 } \]
	or equivalently 
\[ \xymatrix{
	0  \ar[r] 
	& \Z_2 \ar[r] \ar[d]^{\eta} 	
	& \Z_2 \ar[r] \ar[d]^\eta
	& 0 \ar[r] \ar[d]^\eta
	& 0 \\
	0 \ar[r]
	& 0 \ar[r] 
	& \Z_2 \ar[r] 
	& \Z_2 \ar[r] 
	& 0	
}\]
 The vertical $\eta$ maps on the left and right come from $\mathcal A$ as described above.

%

Since the nonzero horizontal maps must be isomorphisms, the commutative diagram forces the vertical map
$ \eta_6 \colon KO_6( C \sp * \sr (\Lambda, \gamma)) \rightarrow KO_{7} (C \sp * \sr (\Lambda, \gamma)) $
in the center of the diagram
to be zero, which contradicts the known value of 
$\eta$ in $KO_*(\mathcal{E}_n)$.
\end{proof}

\section{Existence of a Rank-3 Graph with Involution} \label{construction}

In this section, we will construct a 3-graph $\Lambda_n$ with involution $\gamma_n$, by taking products of 1-graphs with involution, such that $C^*\sr(\Lambda_n, \gamma_n)$ is stably isomorphic to $\mathcal E_n$. 
In what follows, we use the following convention for suspension of graded modules. If $H = \{ H_i \}$ is a $\Z$-graded group, then $\Sigma H$ is a $\Z$-graded group with $(\Sigma H)_i = H_{i+1}$. Similarly, $\Sigma^{-1} H$ is a $\Z$-graded group with $(\Sigma^{-1} H)_i = H_{i-1}$.
This convention is consistent with $K$-theory and suspensions of $C^*$-algebras: $K_*(S^n A) = \Sigma^n K_*(A)$.

For the proof of the following proposition, we will need a few more preliminaries.  
For a graph $\Lambda$, a subset $X\subseteq \Lambda^0$ is {\it hereditary} if whenever $v \in X$ and $v \Lambda w \not= \emptyset$, then $w \in X$.   A {\it cycle} in $\Lambda$ is a path $e_1 e_2 \cdots e_n$ with $r(e_1) = s(e_n)$ but, for all $1\leq i < n$, we have $s(e_i) \not= r(e_{i+1})$.  We say that a cycle {\it has an entrance} if there exists $ 1 \leq j \leq n$ and an edge $f \not= e_j$ with $r(f) = r(e_j)$.   If the only hereditary subsets of $\Lambda^0$ are $\emptyset$ and $\Lambda^0$, and every cycle in $\Lambda$ has an entrance, then  $C^*(\Lambda)$ is simple \cite[Theorem 12]{szyman}. 

\begin{proposition} \label{suspension}
There exists a 1-graph $\Lambda$ and involution $\gamma$ such that $C \sp * \sr (\Lambda, \gamma)$ is simple and purely infinite
and that $K\crr( C \sp * \sr(\Lambda, \gamma)) \cong \Sigma K\crr(\R)$.  
\end{proposition}

\begin{proof}

Let $\Lambda$ be the 1-graph below (which extends infinitely in both directions) and let $\gamma$ be the non-trivial involution, which fixes the vertices and edges of the infinite branch on the left and swaps the vertices and edges of the two infinite branches on the right in the obvious way. (In fact $\gamma$ is the only non-trivial involution on $\Lambda$.)

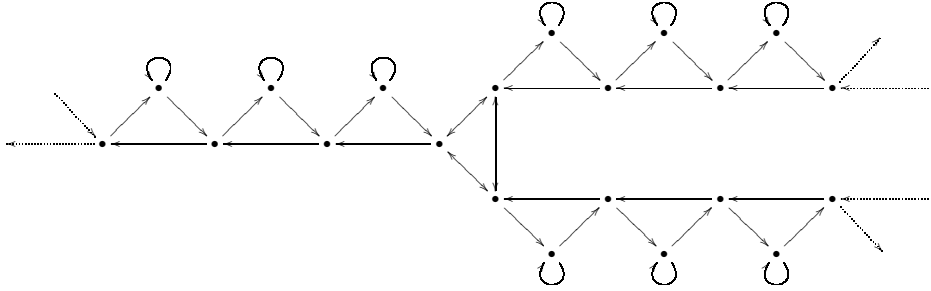
\[
\resizebox{5in}{!}{
\xymatrix{ 
        &&&&&&&&
        && \bullet \ar[dr] \ar@(ur,ul) 
        && \bullet \ar[dr] \ar@(ur,ul) 
        && \bullet \ar[dr] \ar@(ur,ul) 
        && 
        \\
        &
        && \bullet  \ar@(ur,ul)  \ar[dr]
        && \bullet  \ar@(ur,ul)   \ar[dr]
        && \bullet  \ar@(ur,ul)   \ar[dr]
        &&  {\bullet}   \ar@{<->}[dd] \ar[ur]
        && {\bullet} \ar[ll]  \ar[ur] 
        && {\bullet} \ar[ll]  \ar[ur] 
        && {\bullet} \ar[ll] \ar@{<.}[rr]  \ar@{.>}[ur]
        &&    \\
         \ar@{<.}[rr] 
        && \bullet   \ar@{<.}[ul] \ar[ur]
        && \bullet \ar[ll] \ar[ur]
        && \bullet \ar[ll] \ar[ur]
        && \bullet \ar@{<->}[ur] \ar@{<->}[dr]  \ar[ll]
        \\
        &&&&&&&&&  {\bullet} \ar[dr]
        && {\bullet} \ar[ll]  \ar[dr] 
         && {\bullet} \ar[ll]  \ar[dr] 
        && {\bullet} \ar@{<.}[rr]   \ar@{.>}[dr] \ar[ll]
        && 
         \\
        &&&&&&&&
        && \bullet \ar[ur] \ar@(dr,dl) 
        && \bullet \ar[ur] \ar@(dr,dl) 
        && \bullet \ar[ur] \ar@(dr,dl) 
        && 
        } \; }
     \]

\vspace{.5cm}

It is straightforward to check that $\Lambda$ has no nontrivial hereditary subsets, and that every cycle has an entrance, so simplicity of the complex algebra $C^*(\Lambda)$ follows from \cite[Theorem 12]{szyman}. 
As every cycle has an entrance, $\Lambda$ is aperiodic.  One now easily checks 
 that the conditions of  \cite[Theorem~3.9]{kpr} are satisfied, and so $C \sp * (\Lambda)$ is purely infinite.
Consequently,  \cite[Theorem~3.9]{brs} implies that the real $C \sp *$-algebra $C \sp * \sr(\Lambda, \gamma)$ is also simple and purely infinite.

We now show that $KU_0( C \sp * \sr( \Lambda, \gamma)) = 0$ and $KU_1( C \sp * \sr( \Lambda, \gamma)) = \Z$. (This is the same as calculating $K_*( C \sp * (\Lambda))$ and does not involve the involution $\gamma$.)  Let $M$ be the adjacency matrix for $\Lambda$ (so $M_{v,w}$ is the number of edges from $w$ to $v$, as in Equation~\ref{eq:adjacency-matrices}). Then  $KU_0( C \sp * \sr( \Lambda, \gamma)) \cong \coker (I - M^t)$.  That is, 
$KU_0( C \sp * \sr (\Lambda, \gamma))$ is generated by vertex projection classes $[p_v]$, which are subject only to relations of the form
$$ [p_v] = \sum_{w \in \Lambda^0} M_{v,w} [p_w]  \; .$$
Let $v$ be one of the vertices of $\Lambda$ that supports a loop, and let $w\not= v$ be the vertex for which there is an edge from $w$ to $v$ (in each case there is a unique such $w$). Then the formula above implies that
$[p_v] = [p_v] + [p_w]$, and so $[p_w] = 0$. If $[p_w] = 0$ we will say that $w$ is a {\it zero vertex}.
Now if $w$ is a zero vertex and there is only one edge to $w$, say from vertex $u$, then it follows that $u$ is also a zero vertex. 
More generally, if $w$ is a zero vertex and all the edges to $w$ are known to emanate from zero vertices except possibly one edge from vertex $u$, then it follows that $u$ is also a zero vertex.
Using these principles, it is now straightforward to work through the graph and to find that every vertex is a zero vertex. Hence $KU_0( C \sp * \sr( \Lambda, \gamma)) =  0$.

We know that $KU_1( C \sp * \sr( \Lambda, \gamma)) \cong \ker (I - M^t)$, which is to say that 
\begin{equation}
\label{eq:alpha}
KU_1( C \sp * \sr( \Lambda, \gamma))\cong  N_{\Lambda}: =
 \Big\{ \alpha \colon \Lambda^0 \rightarrow \Z \mid  \alpha(v) =  \sum_{w \in \Lambda^0} M_{w,v} \, \alpha(w) \Big\} \; .
 \end{equation}
Let $v$ be one of the vertices of $\Lambda$ that supports a loop, and let $w \not= v $ be the vertex for which there is an edge from $v$ to $w$ (in each case there is a unique such $w$). 
Then for any $\alpha \in N_\Lambda$,
$\alpha(v) = \alpha(v) + \alpha(w)$, which implies that $\alpha(w) = 0$. 
If $\alpha(w) = 0$ for all $\alpha \in N_\Lambda$ we will say that $w$ is a {\it null vertex}.
Now if $w$ is a null vertex and there is only one edge emanating from $w$, say to vertex $u$, then it follows that $u$ is also a null vertex. 
More generally, if $w$ is a null vertex and all the edges from $w$ are known to point to null vertices except possibly one edge to vertex $u$, then $u$ is also a null vertex.
Using these principles, it is now straightforward to work through the graph and to find that every vertex is a null vertex, except for the six vertices labelled $u,v, w, x,y,z$ shown below.  
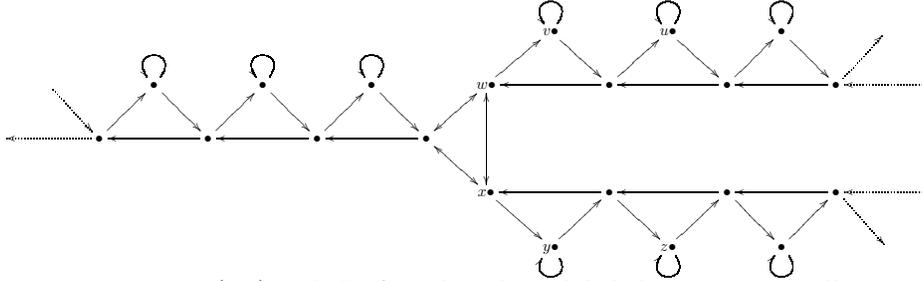
\[
\resizebox{5in}{!}{
\xymatrix{ 
        &&&&&&&&
        && v \bullet \ar[dr] \ar@(ur,ul) 
        && u \bullet \ar[dr] \ar@(ur,ul) 
        && \bullet \ar[dr] \ar@(ur,ul) 
        && 
        \\
        &
        &&  \bullet  \ar@(ur,ul)  \ar[dr]
        && \bullet  \ar@(ur,ul)   \ar[dr]
        && \bullet  \ar@(ur,ul)   \ar[dr]
        &&  w {\bullet}   \ar@{<->}[dd] \ar[ur]
        && {\bullet} \ar[ll]  \ar[ur] 
        && {\bullet} \ar[ll]  \ar[ur] 
        && {\bullet} \ar[ll] \ar@{<.}[rr]  \ar@{.>}[ur]
        &&    \\
         \ar@{<.}[rr] 
        && \bullet   \ar@{<.}[ul] \ar[ur]
        && \bullet \ar[ll] \ar[ur]
        && \bullet \ar[ll] \ar[ur]
        && \bullet \ar@{<->}[ur] \ar@{<->}[dr]  \ar[ll]
        \\
        &&&&&&&&&  {x \bullet} \ar[dr]
        && {\bullet} \ar[ll]  \ar[dr] 
         && {\bullet} \ar[ll]  \ar[dr] 
        && {\bullet} \ar@{<.}[rr]   \ar@{.>}[dr] \ar[ll]
        && 
         \\
        &&&&&&&&
        &&  y \bullet  \ar[ur] \ar@(dr,dl) 
        && z \bullet \ar[ur] \ar@(dr,dl) 
        && \bullet \ar[ur] \ar@(dr,dl) 
        && 
        } \; }
     \]


Using Equation \eqref{eq:alpha} and the fact that the unlabeled vertices are null vertices, we see that any $\alpha \in N_\Lambda$ must satisfy the equations
\begin{align*}
& 0 = \alpha(u) + \alpha(w) & \quad 
\alpha(w) = \alpha(v) + \alpha(x) & \\
& 0 = \alpha(w) + \alpha(x) & \quad 
\alpha(x) = \alpha(w) + \alpha(y) & \quad 0 = \alpha(x) + \alpha(z)
\end{align*}
Solving this system over $\Z$, we find that $\alpha(u)$ is a free variable and that
\[ \alpha(v) = -2 \alpha(u)  ,~ 
 \alpha(w) = -\alpha(u)  ,~ 
 \alpha(x) = \alpha(u)  ,~ 
 \alpha(y) = 2 \alpha(u)  ,~  \text{and}~
 \alpha(z) = -\alpha(u) \; . \]  
Thus $N_\Lambda \cong \Z$.  Hence $KU_* ( C \sp * \sr (\Lambda, \gamma)) = K_*( C \sp * (\Lambda)) = (0, \Z)$. 

Turning to the real $K$-theory, we now prove that $KO_*( C \sp * \sr( \Lambda, \gamma))) = (\Z_2, \Z_2, 0, \Z, 0, 0, 0, \Z)$.
First, we show that the real and complex $E^2 = E^\infty$ page of the Evans spectral sequence for $C^*\sr(\Lambda,\gamma)$ is as follows.
		
\def\vvline{\hfil\kern\arraycolsep\vline\kern-\arraycolsep\hfilneg}
		\vspace{-4mm} 	
\begin{equation}	\begin{array}{ cccc } 
				\multicolumn
				{3}{c}{ \underline{ \text{real part}}} \\
				\vspace{.25cm} \\
				\vdots \vvline& \hspace{.3cm} \vdots \hspace{.3cm} & \hspace{.3cm} \vdots \hspace{.3cm}  \\
				7  \vvline & 0 & 0  \\
				6   \vvline  & 0 & \Z   \\
				5   \vvline  &  0    &  0     \\
				4   \vvline  &  0 & 0  \\
				3    \vvline &  0  & 0    \\
				2   \vvline  & 0 & \Z    \\
				1   \vvline  &   \Z_2  &  0   \\
				0 \vvline & \Z_2 & 0 \\  \hline
				q/p \vvline & 0 & 1 
			\end{array}
			\hspace{3cm}
			\begin{array}{ cccc }
				\multicolumn
				{3}{c}{ \underline{ \text{complex part}}} \\
				\vspace{.25cm} \\
				\vdots  \vvline & \hspace{.3cm} \vdots \hspace{.3cm} &\hspace{.3cm} \vdots \hspace{.3cm}  \\
				7  \vvline & 0 & 0  \\
				6   \vvline  & 0 & \Z \\
				5   \vvline  &  0    &  0       \\
				4   \vvline  &  0 & \Z \\
				3    \vvline &  0  & 0   \\
				2   \vvline  &  0 & \Z \\
				1   \vvline  & 0    &    0    \\
				0   \vvline  & 0 & \Z \\ \hline
				q/p \vvline & 0 & 1 
			\end{array}
\label{eq:spec-seq-real-suspension}
\end{equation}

We have already discussed the complex part of this spectral sequence. For the real part we will only discuss the computations for the rows corresponding to $q = -1, 0, 1$. As we will see, this is enough to determine $KO_*(C \sp * \sr (\Lambda, \gamma))$. The other rows can be computed using similar methods and we include them in the table above for completeness, but we will neither need nor discuss them.

First, the spectral sequence for a 1-graph with involution always vanishes in row $q= -1$, since the chain complex vanishes in that degree. 
To compute row $q = 1$ of the spectral sequence, from Table~3 and Theorem~3.17 of \cite{boersema-gillaspy}, we have that $E^2_{0,1} $ and $E^2_{1,1}$ are the cokernel and kernel of the map
\vspace{-2mm} 
$$ (\partial_1)_1 = I - M^t_{11} \colon  \Z_2^{\Lambda^0_f} \rightarrow \Z_2^{\Lambda^0_f} $$
where $\Lambda^0_f$ is the set of fixed vertices of $(\Lambda, \gamma)$ and $M_{11}$ is the restriction of the incidence matrix to those vertices.
So it suffices to consider the graph consisting of the fixed points of $\Lambda$, shown here.
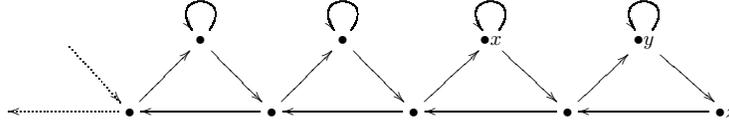
\[
\resizebox{4in}{!}{
\xymatrix{ 
               &
        && \bullet  \ar@(ur,ul)  \ar[dr]
        && \bullet  \ar@(ur,ul)   \ar[dr]
          &&  \bullet x  \ar@(ur,ul)   \ar[dr]
        &&  \bullet y  \ar@(ur,ul)   \ar[dr]
         \\
                \ar@{<.}[rr] 
        && \bullet  \ar@{<.}[ul] \ar[ur]
        && \bullet \ar[ll] \ar[ur]
        && \bullet \ar[ll] \ar[ur]
          && \bullet \ar[ll] \ar[ur]
        &&  \bullet z  \ar[ll]
               } \; }
     \]
Using this graph, and the same sort of analysis that we did in the complex case, we find that
$\coker (I - M^t_{11}) = \Z_2$. More precisely, working modulo 2 we find that $[p_v] = 0$ for all vertices $v$ in $\Lambda^0_f$ except those labeled $x,y$ and $z$ in the graph above and that $[p_x] = [p_y] = [p_z] \neq 0$. We also find easily that $\ker (I - M^t_{11}) = 0$.

Now, for $q = 0$, 
let $\Lambda^0_f$ be the set of fixed vertices (the branch on the left of $\Lambda)$, $\Lambda^0_g$ be the set of vertices of the ``upper right" branch of $\Lambda$, and $\Lambda^0_h$ be the set of vertices of the ``lower right" branch.
With this structure on $\Lambda^0$, the (infinite) matrix $B = I - M^t$ can be written in block form as 
\[B = I - M^t= \begin{pmatrix} B_{11} & B_{12} & B_{12} \\ B_{2 1} & B_{22} & B_{23} \\ B_{21} & B_{23} & B_{33} \end{pmatrix} \;  \]
where, for example, $B_{12}$ keeps track of edges from vertices in $\Lambda^0_f$ to vertices in $\Lambda^0_g$.
From Table~3 of \cite{boersema-gillaspy}, we have that 
$(\partial_1)_0 : \Z^{\Lambda^0_f} \oplus \Z^{\Lambda^0_g} 
		\rightarrow \Z^{\Lambda^0_f} \oplus \Z^{\Lambda^0_g}$ is given by 
\[ (\partial_1)_0 =  \begin{pmatrix} B_{11} & 2B_{12} \\ B_{21} & B_{22} + B_{23} \end{pmatrix} \; .\]

We will use a new graph $\Lambda'$ to analyze this map. 
The graph $\Lambda'$, shown below, is obtained from $\Lambda$ by keeping the vertices from $ \Lambda^0_f$ and $ \Lambda^0_g$. 
For each edge in $\Lambda$ from a vertex in $\Lambda^0_g$ to a vertex in 
$\Lambda^0_f$, we create a corresponding edge in $\Lambda'$ and for each edge in $\Lambda$ from a vertex in $\Lambda^0_f$ to a vertex in $\Lambda^0_g$ we create 2 corresponding edges in $\Lambda'$. 
Also, for each edge from a vertex $v = \gamma(u) \in \Lambda^0_h$ to a vertex $w\in \Lambda^0_g$ we obtain an edge in $\Lambda'$ from $u$ to $w$. 

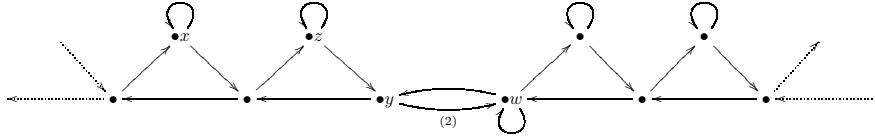
\[ \resizebox{5in}{!}{
\xymatrix{   
        &
        && \bullet x \ar@(ur,ul)  \ar[dr] \ar@{<-}[dl]
        && \bullet z  \ar@(ur,ul)   \ar[dr] \ar@{<-}[dl]
        &&&&  \bullet   \ar[dr] \ar@(ur,ul) 
        && \bullet  \ar[dr] \ar@(ur,ul)  
        && && \\
         \ar@{<.}[rr] 
        && \bullet  \ar@{<.}[ul]
        && \bullet \ar[ll] 
        && \bullet y \ar@/_/[rr]_{(2)} \ar[ll]
        && \bullet  w   \ar@(dr, dl) \ar@/_/[ll] \ar[ur]
        && \bullet   \ar[ll] \ar[ur]
        && \bullet  \ar[ll] \ar@{.>}[ur] \ar@{<.}[rr]
        &&
        } \; }
     \]
\vspace{.5cm}
By construction the adjacency matrix $M'$ for the graph $\Lambda'$ satisfies
\vspace{-5mm} 
$$ I - (M')^t = \begin{pmatrix} B_{11} & 2B_{12} \\ B_{21} & B_{22} + B_{23} \end{pmatrix} \; . $$
Therefore, we can use the graph $\Lambda'$ to find the cokernel and kernel of $(\partial_1)_0$. 
Using the same logic and terminology we used when calculating the complex $K$-theory, we see that $w$ is a zero vertex because it emits an edge to a vertex $v$ which supports a loop,  and the edge from $w$ to $v$ is the only non-loop edge which points to $v$. Indeed, every vertex along the bottom row of the graph $\Lambda'$ except $y$ is a zero vertex.  The fact that these zero vertices (with the exception of $w$) only receive edges from one (potentially) nonzero vertex on the top row of $\Lambda'$ implies that every vertex in the top row except $x$ and $z$ are also zero vertices. Now, $w$ is a zero vertex, but since there are two edges from $y$ to $w$ we obtain the relation
$[p_w] = [p_w] + 2 [p_y]$ which implies that $2 [p_y] = 0$, but $[p_y] \neq 0$.
Finally, from the relations $[p_x] = - [p_y]$ and $[p_z] = [p_y]$ we conclude that $\coker(\partial_1)_0 = \Z_2$.

To compute $\ker (\partial_1)_0$, we also proceed as in the computations for the complex case: All of the vertices in the bottom row of $\Lambda'$, save $w$, are null vertices. Moreover, if a null vertex $v$ emits $n$ edges to a single potentially non-null vertex $u$, we must have $n \alpha(u) = 0$ for any $\alpha \in N_{\Lambda'}$, and as $\alpha(u) \in \Z$ we conclude that $u$ must also be null.  It follows that $w$ is null, as are all of the vertices in the top row of $\Lambda'$.  That is, $\ker (\partial_1)_0 = \{0\}$. 

Now, with the three rows of that we've identified, the spectral sequence \eqref{eq:spec-seq-real-suspension} implies that   
$KO_0(C \sp * \sr (\Lambda, \gamma)) \cong KO_1(C \sp * \sr (\Lambda, \gamma)) \cong \Z_2$.
We claim that using this we can compute $KO_i(C \sp * \sr (\Lambda, \gamma))$ for $2 \leq i \leq 7$
using the long exact sequence \eqref{eq:LES}
and other aspects of $\CR$-structure.
The fact that $KU_6(C^*\sr(\Lambda, \gamma)) = KU_0(C \sp * \sr (\Lambda, \gamma)) = 0$ implies that $\eta_0$ is  injective, and hence an isomorphism, and also that  $\eta_{-1}$ is surjective. Since $KU_2(C \sp * \sr (\Lambda, \gamma)) = 0$ it follows that $\eta_1$ is surjective. 
Thus if $KO_2(C \sp * \sr (\Lambda, \gamma))$ has a non-zero element, then it would have to be in the image of $\eta_1 \circ \eta_0 \circ \eta_{-1} =\eta^3$. But $ \eta^3 = 0$ for all real $C \sp *$-algebras. Thus $KO_2(C \sp * \sr (\Lambda, \gamma)) = 0$.

Since $KO_2(C \sp * \sr (\Lambda, \gamma)) = 0$, the long exact sequence implies that $c_3 \colon KO_3(C \sp * \sr (\Lambda, \gamma)) \rightarrow KU_3(C \sp * \sr (\Lambda, \gamma)) = \Z$ is injective.  This forces $KO_3(C^*\sr(\Lambda, \gamma)) = \Z$.  Moreover, as $\text{im} \, r_1 = \ker \eta_1 = \Z_2$ we must have $r_1: \Z \to \Z_2$ the unique nonzero map.  Since $\text{im}\, c_3 \cong \ker  r_1$, we conclude that $c_3$ is multiplication by 2.  The relation $rc =2$ then implies that $r_3 = 1$.

Continuing this process using the long exact sequence, we compute that $KO_*( C \sp * \sr( \Lambda, \gamma))) = (\Z_2, \Z_2, 0, \Z, 0, 0, 0, \Z)$.  The module maps $\eta, r, c, \psi$ are then completely determined by these groups and the long exact sequence \eqref{eq:LES}; that is,  $K\crr(C^*\sr(\Lambda, \gamma))$ and hence $K\crt(C^*\sr(\Lambda, \gamma))$ coincide with $\Sigma K\crt(\R)$.
\end{proof}

\begin{lemma} \label{products}
Suppose that $(\Lambda_1, \gamma_1)$ and $(\Lambda_2, \gamma_2)$ are higher-rank graphs with involutions. Then $(\Lambda_1 \times \Lambda_2, \gamma_1 \times \gamma_2)$ is  a higher-rank graph with involution and
\[ C \sp * \sr( \Lambda_1 \times \Lambda_2, \gamma_1 \times \gamma_2)
	\cong C \sp * \sr(\Lambda_1, \gamma_1) \otimes\sr C \sp * \sr(\Lambda_2, \gamma_2) \; . \]
\end{lemma}

\begin{proof}
Assume that $\Lambda_1$ and $\Lambda_2$ have rank $k_1$ and $k_2$ respectively.
From \cite[Proposition~1.8]{kp} the product $\Lambda_1 \times \Lambda_2$ is a graph of rank $k_1 + k_2$, with degree functor
$d(\lambda_1, \lambda_2) = d(\lambda_1) + d(\lambda_2)$. Furthermore, there is an involution $\gamma$ on $\Lambda_1 \times \Lambda_2$  defined by
$\gamma( \lambda_1, \lambda_2) = (\gamma_1(\lambda_1), \gamma_2(\lambda_2) )$.

From \cite[Corollary~3.5]{kp}, there is an isomorphism $\phi \colon C \sp *(\Lambda_1 \times \Lambda_2) \rightarrow C \sp * (\Lambda_1) \otimes C \sp * (\Lambda_2)$ defined by $\phi( s_{(\lambda_1, \lambda_2)}) = s_{\lambda_1} \otimes s_{\lambda_2}$. To finish the proof, we need only show that $\phi$ preserves the real structures \eqref{eq:real-C*-alg}
of $C \sp *(\Lambda_1 \times \Lambda_2) $ and $ C \sp * (\Lambda_1) \otimes C \sp * (\Lambda_2)$ which are induced by the graphical involutions $\gamma_i$.  This is straightforward:
\begin{align*}
\phi (\widetilde \gamma(s_{(\lambda_1, \lambda_2)} )
	&=\phi (s^*_{(\gamma_1(\lambda_1),  \gamma_2(\lambda_2))} )) 
	= s^*_{\gamma_1(\lambda_1)} \otimes s^*_{\gamma_2(\lambda_2)} 
	= \widetilde{\gamma_1}( s_{\lambda_1}) \otimes \widetilde{\gamma_2}( s_{\lambda_2}) \\
	&=\widetilde{\gamma_1 \otimes \gamma_2}( s_{\lambda_1} \otimes s_{\lambda_2})   
	= \widetilde{\gamma_1 \otimes \gamma_2}(\phi( s_{(\lambda_1, \lambda_2)} ) ) \; . \qedhere
\end{align*}
\end{proof}

\begin{theorem} \label{rank3existance} Let $k = 2n+1$ be an odd integer, $k \geq 3$. There exists a rank-3 graph with involution $(\Lambda_n, \gamma_n)$ such that $C\sp * \sr (\Lambda_n, \gamma_n) \cong \mathcal{K}\sr \otimes\sr \E_{2n+1}$. 
Furthermore, there exists a projection $p \in C\sp * \sr(\Lambda_n, \gamma_n)$ such that $p C\sp * \sr(\Lambda_n, \gamma_n) p \cong \E_{2n+1} $.
\end{theorem}

\begin{proof}
Let $(\Lambda, \gamma)$ be the 1-graph given by Proposition~\ref{suspension} and let $(E_n, \ve_n)$ be the finite 1-graph with involution from Example~6.2 of \cite{boersema-MJM}. Then both
$C \sp * \sr (\Lambda, \gamma)$ and
$C\sp * \sr (E_n, \ve_n)$ are simple and purely infinite and we have
\[ K\crr( C\sp * \sr (\Lambda, \gamma)) \cong \Sigma K\crr(\R)
	\quad \text{and} \quad
		K\crr( C\sp * \sr (E_n, \ve_n)) \cong \Sigma^6 K\crr( \E_n) \; . \]
This implies by \cite[Proposition~2.1]{brs} that
\[ K\crt( C\sp * \sr (\Lambda, \gamma)) \cong \Sigma K\crt(\R)
	\quad \text{and} \quad
		K\crt( C\sp * \sr (E_n, \ve_n)) \cong \Sigma^6 K\crt( \E_{2n+1}) \; . \]

Let $(\Lambda_n, \gamma_n)$ be the product rank-3 graph with involution
$ (\Lambda_n, \gamma_n) = 
	(\Lambda, \gamma)  \times (\Lambda, \gamma) 
		 \times  (E_n, \ve_n) . $
Lemma \ref{products} then implies that 
\[ C \sp * \sr (\Lambda_n, \gamma_n) 
	\cong C \sp * \sr(\Lambda, \gamma) \otimes \sr C \sp * \sr(\Lambda, \gamma)
		  \otimes\sr  C \sp * \sr(E_n, \ve_n).  \]
Now, $K\crt( C \sp * \sr (\Lambda, \gamma))$ is a free $\CRT$-module, since it is isomorphic to a suspension of $K\crt(\R)$ (see \cite[Section 2.1]{boersema2002}). Therefore the K\"unneth formula for the $K$-theory of real $C \sp*$-algebras (Proposition~3.5 and Theorem~4.2 of \cite{boersema2002}) gives
\begin{align*} K\crt( C \sp * \sr (\Lambda_n, \gamma_n))  
&\cong 
	  K\crt( C \sp * \sr(\Lambda, \gamma)) \otimes_\mathcal{CRT} K\crt( C \sp * \sr(\Lambda, \gamma))
		  \otimes_\mathcal{CRT} K\crt( C \sp * \sr(E_{2n+1}, \varepsilon_n))  \\
&\cong  \Sigma^{2}  K\crt( \R)
	\otimes_\mathcal{CRT}  \Sigma^6 K\crt(\E_{2n+1}) \\
&\cong K\crt(\E_{2n+1}) \; .
\end{align*}

As $|\Lambda_n^0| = \infty$, $C^*\sr(\Lambda_n, \gamma_n)$ cannot be unital. Being a tensor product of  simple, purely infinite, nuclear $C^*$-algebras (thanks to Proposition \ref{suspension} and \cite[Example 6.2]{boersema-MJM}), $C \sp * (\Lambda_n, \gamma_n)$ is also simple and  purely infinite (cf.~\cite[Proposition 4.1.8(iii)]{rordam-classification-monograph}).  It follows from Zhang's dichotomy \cite{zhang-purely-infinite} that $C^*(\Lambda_n, \gamma_n)$ is stable.  Consequently, $C^*\sr(\Lambda_n, \gamma_n)$ is a stable, simple, purely infinite, real $C \sp *$-algebra, because its complexification  is simple and purely infinite (see Theorem~3.9 of \cite{brs}).
By the same token, $\mathcal{K}\sr \otimes\sr \E_{2n+1}$ is a a stable, simple, purely infinite, real $C \sp *$-algebra. 
 Thus the first statement of the theorem follows by the classification of real Kirchberg algebras, \cite[Theorem 10.2, Part (1)]{brs}.

To prove the second statement, by \cite[Proposition~3.13]{brs} there is a projection $p \in C \sp * \sr(\Lambda_n, \gamma_n)$
such that $[p]$ is a generator of $KO_0(C \sp * \sr(\Lambda_n, \gamma_n) ) = \Z_{2n}$.
Then 
\[K\crr( p C \sp * \sr(\Lambda_n, \gamma_n) p) \cong K\crr( C \sp * \sr(\Lambda_n, \gamma_n)) \cong K\crr(\mathcal{E}_{2n+1}) \] 
 (where the first isomorphism is by \cite[Proposition~9]{boersema-JFA}).
Furthermore the class of the identity $[p] \in KO_0( pC \sp * \sr(\Lambda, \gamma )p ) \cong \Z_{2n}$ corresponds under this isomorphism to the class of the identity
$[1] \in KO_0( \mathcal{E}_{2n+1}) \cong \Z_{2n}$. Therefore by \cite[Theorem 10.2, Part (2)]{brs}, we have
$p C \sp * \sr(\Lambda, \gamma) p \cong \mathcal{E}_{2n+1}$.
\end{proof}


\section{Appendix -- real Kirchberg suspension algebras}
\label{sec:appendix}

In the previous section, we introduced a graph with involution $(\Lambda, \gamma)$ for which 
$K \crr( C \sp * (\Lambda, \gamma)) \cong \Sigma K\crr(\R)$. We consider this algebra as a sort of real Kirchberg suspension, since it is a real purely infinite simple stable nuclear $C \sp *$-algebra satisfying the UCT, and with the same $KK$-type as the suspension algebra 
$S \R  \cong C_0((0,1), \R) .$
By repeatedly taking the product of this graph with itself, which corresponds to repeatedly tensoring this algebra with itself, we can obtain a rank-$i$ graph, the real $C \sp *$-algebra of which is a real Kirchberg algebra with the same $KK$-type as $S^i \R$, for any $i$. 
 It is natural to ask which of these suspensions can be obtained from a 1-graph with involution. In this section, we will answer this question completely, providing a full characterization of the integers $i$ (mod 8) for which there exists a 1-graph with involution $(\Lambda, \gamma)$ such that $K\crr( C \sp * (\Lambda, \gamma)) \cong \Sigma^i K\crr( \R)
 \cong K\crr( S^i \R)$. 

\begin{proposition}
For each $ j = \{ -2, -1, 0, 1\}$ there exists a 1-graph with involution $(\Lambda, \gamma)$ such that
$K\crr( C \sp * (\Lambda, \gamma)) \cong \Sigma^i K\crr( \R)$. 
Furthermore, $C \sp * \sr (\Lambda, \gamma)$ is simple and purely infinite.
\end{proposition}

\begin{proof}[Sketch of proof.]
For $i = 1$, this follows from Proposition~\ref{suspension}.
For each $i \in \{-2, -1, 0\}$ we show below a graph or graph with involution that satisfies
$K\crr( C \sp * (\Lambda, \gamma)) \cong \Sigma^i K\crr( \R)$.
The $K$-theory calculations, not shown, are carried out using the same techniques as in the proof of Proposition \ref{suspension}.

{\underline{$i = 0$}}. ~~ The graph $\Lambda$ is shown below, with trivial involution $\gamma = \id$.

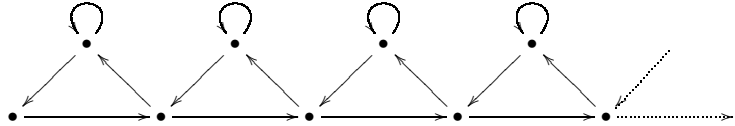
\begin{equation*}
\resizebox{4in}{!}{
\xymatrix{ 
& \bullet \ar[dl] \ar@(ur,ul) \ar@{<-}[dr]
&& \bullet \ar[dl] \ar@(ur,ul) \ar@{<-}[dr]
&& \bullet \ar[dl] \ar@(ur,ul) \ar@{<-}[dr]
&& \bullet \ar[dl] \ar@(ur,ul) \ar@{<-}[dr]
&& \\
  {\bullet}   \ar[rr]
&& {\bullet}  \ar[rr] 
&& {\bullet} \ar[rr]   
&& {\bullet} \ar[rr]  
&& {\bullet}    \ar@{.>}[rr] \ar@{<.}[ur]
&&\\
} \; }
\end{equation*}
\vspace{.25cm}

{\underline{$i = -1$}.} ~~ The graph $\Lambda$ is shown below, with trivial involution $\gamma = \id$.

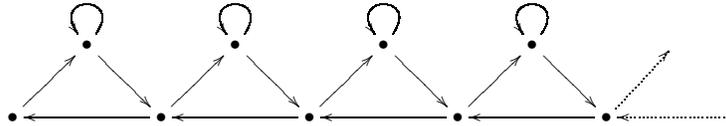
\begin{equation*}
\resizebox{4in}{!}{
\xymatrix{ 
& \bullet \ar[dr] \ar@(ur,ul)
&& \bullet \ar[dr] \ar@(ur,ul)
&& \bullet \ar[dr] \ar@(ur,ul)
&& \bullet \ar[dr] \ar@(ur,ul)
&& \\
  {\bullet}  \ar[ur]
&& {\bullet}  \ar[ll] \ar[ur] 
&& {\bullet} \ar[ll]  \ar[ur] 
&& {\bullet} \ar[ll]  \ar[ur]
&& {\bullet}  \ar[ll]   \ar@{<.}[rr] \ar@{.>}[ur]
&&\\
} \; }
\end{equation*}
\vspace{.25cm}

{\underline{$i = -2.$}} The graph $\Lambda$ is shown below, with the  non-trivial involution $\gamma$ which interchanges the right-hand branches.

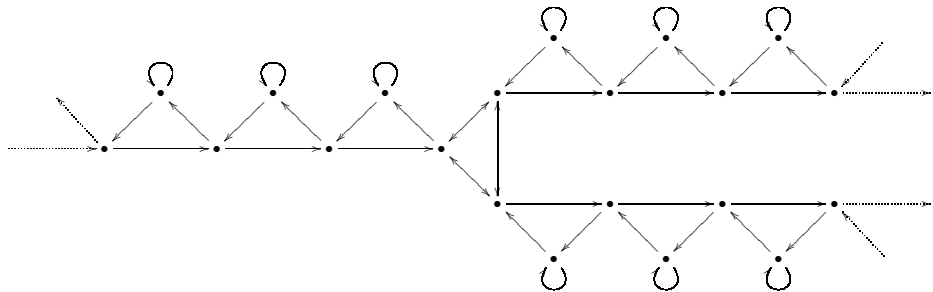
\[
\resizebox{5in}{!}{
\xymatrix{ 
        &&&&&&&&
        &&  \bullet \ar@{<-}[dr] \ar[dl] \ar@(ur,ul) 
        &&  \bullet \ar@{<-}[dr] \ar[dl] \ar@(ur,ul) 
        &&  \bullet \ar@{<-}[dr] \ar[dl] \ar@(ur,ul) 
        && 
        \\
        &
        &&  \bullet  \ar@(ur,ul)  \ar[dl] \ar@{<-} [dr]
        && \bullet  \ar@(ur,ul)   \ar[dl] \ar@{<-} [dr]
        && \bullet  \ar@(ur,ul)   \ar[dl] \ar@{<-} [dr]
        &&    {\bullet}   \ar@{<->}[dd] \ar[rr]
        && {\bullet} \ar[rr]
        && {\bullet} \ar[rr] 
        && {\bullet} \ar@{.>}[rr]  \ar@{<.}[ur]
        &&    \\
         \ar@{.>}[rr] 
        && \bullet  \ar[rr] \ar@{.>}[ul] 
        && \bullet \ar[rr] 
        && \bullet \ar[rr] 
        && \bullet \ar@{<->}[ur] \ar@{<->}[dr]  
        \\
        &&&&&&&&&  { \bullet} \ar[rr]
        && {\bullet} \ar[rr]
         && {\bullet} \ar[rr]
        && {\bullet} \ar@{.>}[rr]   \ar@{<.}[dr] 
        && 
         \\
        &&&&&&&&
        &&   \bullet \ar@{<-}[ur]  \ar[ul] \ar@(dr,dl) 
        && \bullet \ar@{<-}[ur]  \ar[ul] \ar@(dr,dl) 
        && \bullet \ar@{<-}[ur]  \ar[ul] \ar@(dr,dl) 
        && 
        } \; }
     \]
    \vspace{.25cm}


    
For the $i= -2$ graph, one can determine all of the groups $KO_i(C^*\sr(\Lambda, \gamma))$ from the associated spectral sequence, except for $KO_2(C^*\sr(\Lambda, \gamma))$.   In that case, the spectral sequence $KO_2( C \sp * \sr (\Lambda, \gamma))$ has the filtration 
$0 \rightarrow \Z \rightarrow KO_2( C \sp *  \sr(\Lambda, \gamma)) \rightarrow \Z_2 \rightarrow 0$. Although this filtration by itself does not determine $KO_2(C^*\sr(\Lambda, \gamma))$, the long exact sequence \eqref{eq:LES} forces $KO_2( C \sp * \sr (\Lambda, \gamma)) = \Z$. 
\end{proof}

\begin{proposition}
For $2 \leq i \leq 5$, there does not exist a 1-graph $(\Lambda, \gamma)$ with involution such that $K\crr(C \sp * \sr (\Lambda, \gamma)) \cong \Sigma^i K\crr(\R)$.
\label{prop:last}
\end{proposition}

\begin{proof}
Suppose that $(\Lambda, \gamma)$ is a graph with involution and $K\crr(C \sp * \sr (\Lambda, \gamma)) \cong \Sigma^i K\crr(\R)$. 

The real Pimsner-Voiculescu sequence (or equivalently, the real Evans spectral sequence) for $K\crr( C \sp * (\Lambda, \gamma))$ implies that $KO_{-1}( C \sp * \sr (\Lambda, \gamma))$ and $KO_{-3} (C \sp * \sr( \Lambda, \gamma))$ are free abelian groups. However, $KO_1(\R) \cong KO_2(\R) \cong \Z_2$. Thus the group $(\Sigma^{2} KO(\R))_{-1} = KO_1(\R)$ has torsion, implying that 
$K\crr( C \sp * \sr(\Lambda, \gamma)) \ncong \Sigma^{2} KO\crr(\R)$, hence $i \neq 2$. Similarly, the groups $(\Sigma^3 KO (\R))_{-1}$, $(\Sigma^4 KO(\R))_{-3} $, and $(\Sigma^5 KO (\R))_{-3} $ have torsion,  showing that $i \neq 3,4,5$.
\end{proof}

\bibliographystyle{amsalpha}
\bibliography{exotic}
\end{document}